\documentclass[12pt]{amsart}
\usepackage{amsthm,amsmath,amssymb,enumerate}
\usepackage{kbordermatrix}
\newtheorem{thm}{Theorem}[section]

\newtheorem{prop}[thm]{Proposition}

\newtheorem{lem}[thm]{Lemma}

\theoremstyle{definition}

\newcommand{\C}{\mathbb{C}}
\newcommand{\Z}{\mathbb{Z}}
\newcommand{\Q}{\mathbb{Q}}

\newcommand{\F}{\mathbb{F}}
\newcommand{\M}{\mathrm{M}}
\newcommand{\bs}{\boldsymbol}
\newcommand{\s}{\sigma}

\newcommand{\chara}{\operatorname{char}}

\begin{document}
\title{Modular Terwilliger algebras of association schemes}
\author{Akihide Hanaki}
\address{Department of Mathematics, Faculty of Science, Shinshu University, Matsumoto 390-8621, Japan}
\email{hanaki@shinshu-u.ac.jp}

% \date{}
 % Association schemes, strongly regular graphs
\thanks{This work was supported by JSPS KAKENHI Grant Number JP17K05165.}

%\dedicatory{Dedicated to Professor Eiichi Bannai and Professor Hikoe Enomoto on their 75-th birthday}

\maketitle

\begin{abstract}
  We define modular Terwilliger algebras of association schemes,
  Terwilliger algebras over a positive characteristic field,
  and consider basic properties.
  We give a condition for the modular Terwilliger algebra to be non-semisimple.
  We show that the dimension of a Terwilliger algebra of a Johnson scheme 
  depends on the characteristic of the coefficient field.
  We also give some other examples.
\end{abstract}

\section{Introduction}
In a series of papers \cite{Terwilliger1992,Terwilliger1993a,Terwilliger1993b},
P.~Terwilliger defined and studied subconstituent algebras of commutative association schemes
for the theory of distance regular graphs.
Now the algebras are called Terwilliger algebras.
Terwilliger algebras are finite-dimensional semisimple algebras over the complex number field.
Since the Terwilliger algebra is defined for an association scheme and a fixed point,
we can expect that it has more combinatorial information than the adjacency algebra.
The Terwilliger algebra is defined as a matrix algebra generated by some matrices
all whose entries are in $\{0,1\}$.
Thus, for a commutative ring $R$ with the identity $1$,
we can define an $R$-algebra generated by these matrices.
We call it the Terwilliger algebra over $R$.
Especially, we call it the modular Terwilliger algebra
when $R$ is a field of positive characteristic.

Let $(X,S)$ be an association scheme, $x\in X$, and $K$ be a field.
We denote by $KT(x)$ the Terwilliger algebra of $(X,S)$ at the point $x$ over $K$.
If $K$ is of positive characteristic, $KT(x)$ is not necessary semisimple.
A natural question is when it is semisimple.
It is known that the dimension of $\C T(x)$ is depending on the choice of the point $x$
\cite{Tomiyama-Yamazaki1994}.
By example, we can see that the semisimplicity of $KT(x)$ also depends on the choice of the point $x$.
Thus the question seems to be difficult, in general.
We will give a sufficient condition for $KT(x)$ to be non-semisimple in Theorem \ref{ssthm},
$KT(x)$ is not semisimple if the characteristic of $K$ divides a valency of some element of $S$.

In Section \ref{sec:Johnson}, we will consider Johnson schemes $J(n,2)$.
Let $K$ be a field of characteristic $p$.
We will see that the dimension of the Terwilliger algebra is $16$ if $p=0$ or $p\nmid n-4$
and $15$ if $p\mid n-4$.
These examples give a negative answer to Terwilliger's small question
in \cite[Conjecture 10]{Terwilliger1993b}.
In Section \ref{sec:ex}, we will give some other examples.

%%%%%%%%%%%%%%%%%%%%%%%%%%%%%%%%%%%%%%%%%%%%%%%%%%%%%%
\section{Preliminaries and definitions}
Let $X$ be a finite set, and let $R$ be a commutative ring with the identity $1$.
We denote by $\M_X(R)$ the full matrix ring over $R$, rows and columns of whose matrices are
indexed by the set $X$.
For $\sin \M_X(R)$, we denote the transposed matrix of $\s$ by $\s^T$.

For $s\subset X\times X$, the \emph{adjacency matrix} $\s_s\in M_X(\Z)$ is defined by
$(\s_s)_{xy}=1$ if $(x,y)\in s$ and $0$ otherwise.
We often regard $\s_s$ is in $\M_X(R)$ for a suitable $R$.

For $s\subset X\times X$, we set $s^*:=\{(y,x)\mid (x,y)\in s\}$.
Clearly we have $\s_{s^*}=\s_s^T$.
For $s\subset X\times X$ and $x\in X$, we set $xs:=\{y\in X\mid (x,y)\in s\}$ and
$sx:=\{y\in X\mid (y,x)\in s\}$.

\subsection{Association schemes}
Let $X$ be a finite set, and let $X\times X=\bigcup_{s\in S}s$ be a partition of $X\times X$.
We call the pair $(X,S)$ an  \emph{association scheme} if
\begin{enumerate}
  \item $1:=\{(x,x)\mid x\in X\}\in S$,
  \item $s^*:=\{(y,x)\mid (x,y)\in s\}\in S$ if $s\in S$,
  \item for $s,t,u\in S$, there is a non-negative integer $p_{st}^u$ such that
  $p_{st}^u=|xs\cap ty|$ when $(x,y)\in u$.
\end{enumerate}
The condition (3) means that $\s_s\s_t=\sum_{u\in S}p_{st}^u\s_u$ by the usual matrix multiplication.
For $s\in S$, $n_s:=p_{ss^*}^1=|xs|$ is independent of the choice of $x\in X$,
and we call this number the \emph{valency} of $s$.

We say an association scheme $(X,S)$ is \emph{commutative} if $p_{st}^u=p_{ts}^u$ for all $s,t,u\in S$,
\emph{symmetric} if $s^*=s$ for all $s\in S$.
Symmetric association schemes are commutative.

By the condition (3), $RS:=\bigoplus_{s\in S}R\s_s$ is an $R$-algebra.
We call $RS$ the \emph{adjacency algebra} of $(X,S)$ over $R$.

It is well known that strongly regular graphs correspond to symmetric association schemes with $|S|=3$.
We often identify them.

\subsection{Terwilliger algebras}
Let $(X,S)$ be an association scheme.
Fix $x\in X$.
We have a partition $X=\bigcup_{s\in S}xs$.
We set the diagonal matrix $E^*_s\in \M_X(\Z)$ whose
$(y,y)$-entry is $1$ if $y\in xs$ and $0$ otherwise.

%Easily we can see that $\Z T(x)=\Z\langle E_s^*\s_tE_u^* \mid s, t, u \in S \rangle$.
Let $R$ be  a commutative ring with $1$.
We regard $\s_s$ and $E_s^*$ are elements in $\M_X(R)$ and set $RT(x)$ the $R$-algebra generated by
$\{E_s^*\s_tE_u^* \mid s, t, u \in S \}$.
We call $RT(x)$ the \emph{Terwilliger algebra} of $(X,S)$ over $R$ at $x$.
The original definition of a Terwilliger algebra in \cite{Terwilliger1992} is $\C T(x)$.
When $R$ is a field of positive characteristic % or some integral domain,
we call $RT(x)$ a \emph{modular Terwilliger algebra}.
Remark that $E_s^*(RT(x)) E_s^*$ is a subalgebra of $RT(x)$ with the
identity element $E_s^*$.

If $K$ and $K'$ have the same characteristic, then the dimensions of the Terwilliger algebras over them are equal.
Semisimplicity is also depending only on the characteristic of the coefficient field, 
because the algebra is defined over the prime field and the prime field is perfect
(see \cite[Chap.~II, Sect.~5]{NT}, for example).

Let $p$ be a prime number, and let $\F_p$ be a field of order $p$.
By our definition, $\F_p T(x)$ is isomorphic to $\Z T(x)/p(\Q T(x)\cap \M_X(\Z))$.
This is different from $\Z T(x)/p\Z T(x)\cong \F_p\otimes_\Z \Z T(x)$, in general.
We have $\dim_{\F_p}\F_p\otimes_\Z \Z T(x)=\mathrm{rank}_\Z \Z T(x)=\dim_\C \C T(x)$,
but we have many examples such that $\dim_\C \C T(x)\ne \dim_{\F_p}\F_p T(x)$
(see Section \ref{sec:Johnson} and Section \ref{sec:ex}).
It is easy to see that $\Z T(x)$ is a $\Z$-submodule 
of $\Q T(x)\cap M_X(\Z)$ of full rank.
Let $e_1,\dots,e_r$ be the elementary divisors ($r=\dim_\C \C T(x)$).
Then 
\begin{eqnarray*}
  \dim_{\F_p}\F_p T(x) &=& |\{i\mid \text{$p$ is prime to $e_i$}\}|\\
                       &=& \dim_\C \C T(x) -|\{i\mid \text{$p$ divides $e_i$}\}|.
\end{eqnarray*}
Especially,
$\dim_\C \C T(x)= \dim_{\F_p}\F_p T(x)$ if and only if
$p$ does not divide the index $|\Q T(x)\cap M_X(\Z): \Z T(x)|$.

Now, the following proposition holds.

\begin{prop}
  Let $(X,S)$ be an association scheme. We fix $x\in X$.
  Then $\Q T(x)\cap \M_X(\Z)= \Z T(x)$ if and only if
  $\dim_\C \C T(x)= \dim_{K}K T(x)$ for any field $K$.
\end{prop}

%%%%%%%%%%%%%%%%%%%%%%%%%%%%%%%%%%%%%%%%%%%%%%%%%%%%%%%%%%%%%%%%%%%%%%%%%%%
\section{Semisimplicity}
The Terwilliger algebra over the complex number field is semisimple
since it is closed by transposition and complex conjugate.
The Terwilliger algebra over a positive characteristic field is not necessary semisimple.
Moreover it depends on the choice of the point $x\in X$.
For example, $(26,10,3,4)$-strongly regular graphs (see Subsection \ref{subsec26})
give such examples in characteristic $7$ and $11$.

\begin{prop}\label{ss1}
  Let $(X,S)$ be an association scheme and fix $x\in X$.
  Let $K$ be a field.
  If the Terwilliger algebra $KT(x)$ of $(X,S)$ over $K$ is semisimple, then
  $E_s^*(KT(x))E_s^*$ is also semisimple for every $s\in S$.
\end{prop}

\begin{proof}
  Suppose $E_s^*(KT(x))E_s^*$ is not semisimple for some $s\in S$.
  The Jacobson radical $\mathcal{J}$ of $E_s^*(KT(x))E_s^*$ is non-zero.
  Since  $\mathcal{J}$ is nilpotent, assume $\mathcal{J}^\ell=0$.
  It is enough to show that $((KT(x))\mathcal{J}(KT(x)))^\ell=0$.
  For $a_i, b_i\in KT(x)$ and $j_i\in\mathcal{J}$, we have
  \begin{eqnarray*}
    & &(a_1 j_1 b_1)(a_2 j_2 b_2)\dots (a_\ell j_\ell b_\ell)\\
    &=&a_1 E_s^* j_1 (E_s^* b_1a_2 E_s^*) j_2 (E_s^* b_2 a_3 E_s^*) \dots (E_s^* b_{\ell-1} a_\ell E_s^*) j_\ell E_s^*b_\ell\\
    &\in& (KT(x))\mathcal{J}^\ell (K(T(x))=0
  \end{eqnarray*}
  and thus $((KT(x))\mathcal{J}(KT(x)))^\ell=0$. 
\end{proof}

The converse of Proposition \ref{ss1} is not true, in general.
We will give an example of order $15$ in Subsection \ref{ex15.5}.
We could find similar examples also of order $19$, $23$, $27$ and $30$.
The examples are non-symmetric.
The author does not know the converse is true or not for symmetric association schemes.

\renewcommand{\kbldelim}{(}
\renewcommand{\kbrdelim}{)}
For $s,t,u\in S$, write
$$\s_u=\kbordermatrix{
     &  &        & xt      &        & \\
     &  & \vrule &         & \vrule & \\
  \cline{2-6}
  xs &  & \vrule & \s^{st}_u & \vrule & \\
  \cline{2-6}
     &  & \vrule &         & \vrule & 
}.$$

\begin{lem}\label{ss2}
  For $s,t,u\in S$, $\s_{st}^u$ is an incidence matrix of a tactical configuration.
  Every row of $\s_{st}^u$ contains $p_{tu^*}^s$ ones and
  every column of $\s_{st}^u$ contains $p_{su}^t$ ones.
\end{lem}

\begin{proof}
  For $y\in xs$, equivalent to $(x,y)\in s$, the $y$-th row of $\s_{st}^u$ contains
  $$\sharp\{z\in xt\mid (y,z)\in u\}
  =\sharp\{z\in X\mid (x,z)\in t,\ (z,y)\in u^*\}=p_{tu^*}^s$$
  ones.
  Similarly, for $z\in xt$, the $z$-th column of $\s_{st}^u$ contains
  $$\sharp\{y\in xs\mid (y,z)\in u\}
  =\sharp\{y\in X\mid (x,y)\in s,\ (y,z)\in u\}=p_{su}^t$$
  ones.
\end{proof}

\begin{lem}\label{ss3}
  For $s\in S$, $KE_s^* J E_s^*$ is a one-dimensional two-sided ideal of $E_s^* KT(x) E_s^*$,
  where $J$ is the square matrix  all whose entries are one.
\end{lem}

\begin{proof}
  The statement holds by Lemma \ref{ss2}.
\end{proof}

\begin{thm}\label{ss4}
  Let $K$ be a field of positive characteristic $p$.
  Suppose that $p$ divides the valency  $n_s$ for some $s\in S$.
  Then $KT(x)$ is not semisimple.
\end{thm}

\begin{proof}
  Suppose $p$ divides $|xs|$ for $s\in S$.
  Then the ideal $KE_s^* J E_s^*$ of $E_s^* KT(x) E_s^*$ is nilpotent.
  Thus $E_s^* KT(x) E_s^*$ is not semisimple, and so is $KT(x)$ by Proposition \ref{ss1}.
\end{proof}

%%%%%%%%%%%%%%%%%%%%%%%%%%%%%%%%%%%%%%%%%%%%%%%%%%%%%%%%%%%%%%%%%%%%%%%%%%%
\section{Johnson graphs $J(n,2)$}\label{sec:Johnson}
The structures of Terwilliger algebras of Johnson schemes $J(n,k)$ were determined in
\cite{Levstein-Maldonado2007, Lv-Maldonado-Wang2014}.
We focus only on the case $k=2$ and consider their modular Terwilliger algebras.
The structure is independent of the choice of a point $x$,
because the automorphism group acts on points transitively.
We will write $T$ instead of $T(x)$ in this section.

Let $(X,S)$ be the Johnson scheme $J(n,2)$ ($n\geq 5$).
This gives an $(n(n-1)/2, 2(n-2), n-2, 4)$-strongly regular graph.
By definition, we set $X:=\{\{i,j\}\mid 1\leq i < j \leq n\}$,
$S:=\{s_0,s_1,s_2\}$, $s_0:=\{(y,y)\mid y\in X\}$,
$s_1:=\{(y,z)\mid |y\cap z|=1\}$, and $s_2:=\{(y,z)\mid y\cap z=\emptyset\}$.
We often write $1$ instead of $s_1$, and so on.
For example, $\s_1$ is $\s_{s_1}$.
Fix $x:=\{1,2\}$.
We have
\begin{eqnarray*}
  xs_0 &=& \{\{1,2\}\},\\
  xs_1 &=& \{\{1,i\}\mid 2\leq i\leq n\}\cup \{(2,i)\mid 2\leq i\leq n\},\\
  xs_2 &=& \{\{i,j\}\mid 3\leq i<j \leq n\}.
\end{eqnarray*}
Valencies are $n_0=1$, $n_1=2(n-2)$ and $n_2 =(n-2)(n-3)/2$.
We fix $\{1,3\}, \{1,4\}, \dots, \{1,n\}, \{2,3\}, \{2,4\},\dots,\{2,n\}$ the order of $xs_1$.
Then we have
$$\s_1=\left(\begin{array}{c|c|c}
               0 & \bs{j}_{2(n-2)} & \bs{0}_{n_2}\\
               \hline
               \bs{j}_{2(n-2)}^T & \begin{array}{cc}
            J_{n-2}-I_{n-2} & I_{n-2}\\
            I_{n-2} & J_{n-2}-I_{n-2}
                                   \end{array} & D \\
               \hline
               \bs{0}_{n_2}^T & D^T & C
             \end{array}\right),$$
where
$I_{n-2}$ is the identity matrix of degree $n-2$,
$J_{n-2}$ is the square matrix of degree $(n-2)$ all whose entries are one,
$\bs{j}_{2(n-2)}$ is the row vector of degree $2(n-2)$  all whose entries are one,
$\bs{0}_{n_2}$ is the zero row vector of degree $n_2$, 
$C$ is the adjacency matrix of $J(n-2,2)$,
and $D$ is a incidence matrix of a tactical configuration,
every row of $D$ contains $n-3$ ones and
every column of $D$ contains $4$ ones.
We set $\Z T':=\sum_{i, j, k}\Z E_i^*\s_jE_k^*$.
This is not closed by multiplication.
It is not so hard to check that
all products $(E_i^*\s_jE_k^*)(E_{i'}^*\s_{'}jE_{k'}^*)$ are in $\Z T'$
except for 
\begin{eqnarray*}
  (E_1^*\s_1E_1^*)(E_1^*\s_1E_1^*)
  &=&(n-2)E_1^*\s_0 E_1^*+(n-4)E_1^*\s_1 E_1^*\\
  && + 2E_1^*\s_2 E_1^*-(n-4)M,\\
  (E_1^*\s_1E_1^*)(E_1^*\s_2E_1^*)
  &=& E_1^*\s_1E_1^* +(n-4)E_1^*\s_2E_1^*+(n-4)M,\\
  (E_1^*\s_2E_1^*)(E_1^*\s_1E_1^*)
  &=& E_1^*\s_1E_1^* +(n-4)E_1^*\s_2E_1^*+(n-4)M,\\
  (E_1^*\s_2E_1^*)(E_1^*\s_2E_1^*)
  &=& (n-3)E_1^*s_0E_1^*+(n-4)E_1^*\s_1E_1^*-(n-4)M,
\end{eqnarray*} 
where
$$M=\left(\begin{array}{c|c|c}
               0 & \bs{0}_{2(n-2)} & \bs{0}_{n_2}\\
               \hline
               \bs{0}_{2(n-2)}^T & \begin{array}{cc}
            O_{n-2} & I_{n-2}\\
            I_{n-2} & O_{n-2}
                                   \end{array} & O \\
               \hline
               \bs{0}_{n_2}^T & O & O
             \end{array}\right),$$
$O$ are zero matrices.
Now we can see that
$$\Z T=\Z\langle E_i^*\s_jE_k^* \mid 0\leq i, j, k\leq 2 \rangle
=\sum_{i, j, k}\Z E_i^*\s_jE_k^*+\Z(n-4)M.$$
On the other hand,
$$\Q T\cap \M_X(\Z)=\sum_{i, j, k}\Z E_i^*\s_jE_k^*+\Z M
\supsetneq \Z T.$$
This gives a negative answer to Terwilliger's small question ``(is generated by?)'' in
\cite[Conjecture 10]{Terwilliger1993b}.
We can get infinitely many such examples in this way.

\begin{thm}\label{ssthm}
  For the Terwilliger algebra of the Johnson scheme $J(n,2)$ ($n\geq 5$),
  the following statements hold.
  \begin{enumerate}[(1)]
    \item The structure of the Terwilliger algebra does not depend on the choice of the point.
    \item $\Z T=\sum_{i, j, k}\Z E_i^*\s_jE_k^*+\Z(n-4)M$ and
    $\Q T\cap \M_X(\Z)=\sum_{i, j, k}\Z E_i^*\s_jE_k^*+\Z M$.
    \item For a field $K$, $\dim_K K T=16$ if $\chara K=0$ or $\chara K\nmid n-4$,
    and $\dim_K K T=15$ if $\chara K\mid n-4$.
  \end{enumerate}
\end{thm}

\begin{proof}
  Statements (1) and (2) are already proved.
  The statement (3) holds by counting non-zero $E_i^*\s_jE_k^*$ and (2).
\end{proof}

%%%%%%%%%%%%%%%%%%%%%%%%%%%%%%%%%%%%%%%%%%%%%%%%%%%%%%%%%%%%%%%%%%%%%%%%%%%
\section{Examples}\label{sec:ex}
\subsection{The non-symmetric association scheme of order $15$ and rank $3$}\label{ex15.5}
There is a unique non-symmetric association scheme of order $15$ and rank $3$,
that is No.~5 in \cite{Hanaki-HP}.
Set $X=\{1,\dots,15\}$.
The automorphism group acts on $X$ intransitively and the orbits are
$\{1,3,5,8,12,13,15\}$, $\{2,4,6,7,9,10,14\}$, $\{11\}$.
Let $K$ be a field of characteristic $2$.
Then
$$\dim_\Q \Q T(1)=33,\quad \dim_K K T(1)=31,$$
$$\dim_\Q \Q T(2)=\dim_K K T(2)=17,$$
$$\dim_\Q \Q T(11)=17,\quad \dim_K K T(11)=15.$$
For all cases, $K T(x)$ ($x=1,2,11$) are not semisimple,
$$\dim_K J(KT(1))=10,\quad  \dim_K J(KT(2))=4,\quad  \dim_K J(KT(11))=2,$$
where $J(KT(x))$ is the Jacobson radical of $KT(x)$.
However, all $E_i^* KT(x) E_i^*$ ($i=0,1,2$, $x=2,11$) are semisimple.
This shows that the converse of Proposition \ref{ss1} is not true.

\subsection{Chang graphs}
There are four $(28,12,6,4)$-strongly regular graphs.
One is $J(8,2)$ and others are Chang graphs \cite[3.11 (vii)]{BCN}.
In the following table, we will only give dimensions of their Terwilliger algebras $KT(x)$.
\begin{center}
  \begin{tabular}{|c||c|c|c|c|}
    \hline
    $\chara K$ & $J(8,2)$ &Chang1  & Chang2  &Chang3  \\
%    $\chara K$ & $J(8,2)$ &Chang1 ($|Aut|=384$) & Chang2 ($|Aut|=360$) &Chang3 ($|Aut|=96$) \\
    \hline
    \hline
    $0,3,5,7$ & $16$ & $20$, $27$&$23$, $27$ &$23$, $35$\\
    \hline
    $2$ & $15$ & $19$ & $23$ & $23$\\
    \hline
  \end{tabular}
\end{center}
A remarkable fact is that the dimensions are independent of the choice of the points
in characteristic $2$.
For irreducible $\C T(x)$-modules, see \cite{Tomiyama-Yamazaki1994}.

\subsection{$(16,6,2,2)$-strongly regular graphs}
There are two $(16,6,2,2)$-strongly regular graphs \cite[3.11 (vi)]{BCN}.
For them, the automorphism groups act transitively on points
and thus the Terwilliger algebras are independent of the choice of the fixed points.
One of them has
$$\dim_\Q \Q T(x)= \dim_K KT(x)=15$$
and the other has
$$\dim_\Q \Q T(x)=20,\quad  \dim_K K T(x)=19,$$
where $K$ is the field of characteristic $2$.

\subsection{$(26,10,3,4)$-strongly regular graphs}\label{subsec26}
There are ten $(26,10,3,4)$-strongly regular graphs.
In the following table, we will give dimensions of their Terwilliger algebras $KT(x)$.
We use the numbering of them in \cite{Hanaki-HP}.
We will write ``$\dots$'' if the dimensions are same with in characteristic $0$.
\begin{center}
  \begin{tabular}{|c||c|c|c|c|}
    \hline
    $\chara K$ & No.~3 & No.~4 & No.~5 \\
    \hline
    \hline
    $0, 13$ & $19, 24, 28, 31, 39, 47$ & $19, 24, 29, 31, 39, 47$ & $24, 31, 35, 39, 47$ \\
    \hline
    $2$ & $19, 22, 23, 28, 29, 30$ & $19, 23, 26, 29, 30$ & $22, 27, 29, 30$ \\
    \hline
    $3$ & $\dots$ & $19, 24, 28, 31, 39, 47$ & $\dots$ \\
    \hline
    $5$ & $19, 24, 27, 31, 39, 47$ & $\dots$ & $24, 31, 35, 38, 47$ \\
    \hline
    $7$ & $\dots$ & $\dots$ & $24, 31, 34, 39, 47$ \\
    \hline
    $11$ & $19, 24, 27, 31, 39, 47$ & $\dots$ & $24, 30, 31, 35, 39, 47$ \\
    \hline
  \end{tabular}

  \begin{tabular}{|c||c|c|c|c|c|c|c|c|}
    \hline
    $\chara K$ & No.~6 & No.~7 & No.~8 \\
    \hline
    \hline
    $0, 13$ & $19, 24, 28, 29, 35, 47$ & $19, 24, 28, 29, 35, 47$ & $19, 28, 29, 35, 47$ \\
    \hline
    $2$ & $19, 23, 26, 27, 28, 30$ & $19, 23, 25, 26, 27, 28, 30$ & $19, 26, 28, 30$ \\
    \hline
    $3$ & $19, 24, 28, 35, 47$ & $19, 24, 28, 35, 47$ & $19, 28, 35, 47$ \\
    \hline
    $5$ & $\dots$ & $\dots$ & $\dots$ \\
    \hline
    $7$ & $\dots$ & $\dots$ & $\dots$ \\
    \hline
    $11$ & $19, 24, 27, 29, 35, 47$ & $19, 24, 27, 29, 35, 47$ & $19, 27, 29, 35, 47$ \\
    \hline
  \end{tabular}

  \begin{tabular}{|c||c|c|c|c|c|}
    \hline
    $\chara K$ & No.~9 & No.~10 & No.~11 & No.~12 \\
    \hline
    \hline
    $0, 13$ & $31, 35$ & $24, 28$ & $19, 28, 29, 35, 47$ & $28, 29, 35, 47$ \\
    \hline
    $2$ & $29, 30$ & $23, 28$ & $19, 25, 26, 27, 28, 30$ & $26, 27, 28, 30$ \\
    \hline
    $3$ & $\dots$ & $\dots$ & $19, 28, 35, 47$ & $28, 35, 47$ \\
    \hline
    $5$ & $\dots$ & $23, 27$ & $\dots$ & $\dots$ \\
    \hline
    $7$ & $31, 34$ & $\dots$& $\dots$ & $\dots$ \\
    \hline
    $11$ &  $\dots$ & $24, 27$ & $19, 27, 29, 35, 47$ & $27, 29, 35, 47$ \\
    \hline
  \end{tabular}
\end{center}

\bibliographystyle{amsplain}

\begin{thebibliography}{1}

\bibitem{BCN}
A.~E. Brouwer, A.~M. Cohen, and A.~Neumaier, \emph{Distance-regular graphs},
  Ergebnisse der Mathematik und ihrer Grenzgebiete (3) [Results in Mathematics
  and Related Areas (3)], vol.~18, Springer-Verlag, Berlin, 1989.

\bibitem{Hanaki-HP}
A.~Hanaki and I.~Miyamoto, \emph{Classification of association schemes with
  small vertices}, published on web (http://math.shinshu-u.ac.jp/\~{
  }hanaki/as/, https://zenodo.org/record/3627821).

\bibitem{Levstein-Maldonado2007}
F.~Levstein and C.~Maldonado, \emph{The {T}erwilliger algebra of the {J}ohnson
  schemes}, Discrete Math. \textbf{307} (2007), no.~13, 1621--1635.
  %\MR{2316671}

\bibitem{Lv-Maldonado-Wang2014}
B.~Lv, C.~Maldonado, and K.~Wang, \emph{More on the {T}erwilliger algebra of
  {J}ohnson schemes}, Discrete Math. \textbf{328} (2014), 54--62. %\MR{3199816}

\bibitem{NT}
H.~Nagao and Y.~Tsushima, \emph{Representations of finite groups}, Academic
  Press Inc., Boston, MA, 1989.

\bibitem{Terwilliger1992}
P.~Terwilliger, \emph{The subconstituent algebra of an association scheme.
  {I}}, J. Algebraic Combin. \textbf{1} (1992), no.~4, 363--388.

\bibitem{Terwilliger1993a}
\bysame, \emph{The subconstituent algebra of an association scheme. {II}}, J.
  Algebraic Combin. \textbf{2} (1993), no.~1, 73--103.

\bibitem{Terwilliger1993b}
\bysame, \emph{The subconstituent algebra of an association scheme. {III}}, J.
  Algebraic Combin. \textbf{2} (1993), no.~2, 177--210.

\bibitem{Tomiyama-Yamazaki1994}
M.~Tomiyama and N.~Yamazaki, \emph{The subconstituent algebra of a strongly
  regular graph}, Kyushu J. Math. \textbf{48} (1994), no.~2, 323--334.
  %\MR{1294534}

\end{thebibliography}

\providecommand{\bysame}{\leavevmode\hbox to3em{\hrulefill}\thinspace}
\providecommand{\MR}{\relax\ifhmode\unskip\space\fi MR }
\providecommand{\MRhref}[2]{%
  \href{http://www.ams.org/mathscinet-getitem?mr=#1}{#2}
}
\providecommand{\href}[2]{#2}

\end{document}